\documentclass[12pt]{amsart}
\usepackage{amssymb,mathrsfs,color,bbold}

\theoremstyle{plain}
\newtheorem{theorem}{Theorem}[section]
\newtheorem{lemma}[theorem]{Lemma}
\newtheorem{proposition}[theorem]{Proposition}
\newtheorem{corollary}[theorem]{Corollary}

\theoremstyle{definition}

\newtheorem{example}[theorem]{Example}
\newtheorem{problem}[theorem]{Problem}

\newcommand{\term}[1]{{\textit{\textbf{#1}}}}

\newcommand{\abs}[1]{\lvert#1\rvert}
\newcommand{\norm}[1]{\lVert#1\rVert}

\title[Uo-convergence]{Some loose ends on unbounded order convergence}

\date{\today}


\author[H. Li]{Hui Li}
\address{School of Electrical Engineering, Southwest Jiaotong University,
Chengdu, Sichuan,
China, 610000.}
\email{lihuiqc@my.swjtu.edu.cn}

\author[Z. Chen]{Zili Chen}
\address{School of Mathematics, Southwest Jiaotong University, Chengdu, Sichuan,
China, 610000.}
\email{zlchen@home.swjtu.edu.cn}

\begin{document}

\begin{abstract}
The notion of almost everywhere convergence has been generalized to
vector lattices as unbounded order convergence, which proves to be a very
useful tool in the theory of vector and Banach lattices. In this short note, we
establish some new results on unbounded order convergence that tie up some loose
ends. 
In particular, we show that every norm bounded positive increasing net in an
order
continuous Banach lattice is uo-Cauchy and that every uo-Cauchy net in an order
continuous Banach lattice has a uo-limit in the universal completion.
\end{abstract}

\maketitle

\section{Introduction}
The notion of unbounded order convergence (uo-convergence) generalizes the
notion of almost everywhere
convergence from function spaces to vector lattices. 
It was first introduced by Nakano in \cite{N:48} and was later used by DeMarr in
\cite{D:64}. In \cite{W:77}, it was used by Wickstead to characterize discrete
Banach
lattices and order continuous Banach lattices which also have order continuous
dual spaces. Lately, uo-convergence was systematically investigated in
\cite{G:14,GTX:16,GX:14} and applied there to the study of abstract martingales,
Banach lattice geometry, and positive operators. A closely related notion of
unbounded norm convergence is introduced and studied in
\cite{DOT:16,KMT:16,TR:04}.

In this short note, we establish some further results on uo-convergence that tie
up some loose ends. We first study order convergence properties of positive
increasing nets. 
We prove that if the order continuous dual of a Banach lattice $X$ separates
points then every norm bounded positive increasing net in $X$ is uo-Cauchy
(Theorem~\ref{uo-cauchy}).
This implies, in particular, that every norm bounded positive increasing
sequence in a Banach function space over a $\sigma$-finite measure space
converges a.e.~to a \emph{real-valued} function.
We also link universal completeness with uo-convergence. We prove that an
order complete vector lattice with the countable sup property is universally
complete if and only if it is uo-complete, i.e., every uo-Cauchy net is
uo-convergent (Proposition~\ref{uo-complete}).
It follows that if the order continuous dual of a vector lattice $X$ separates
points then every uo-Cauchy net in $X$
has a uo-limit in
the universal completion $X^{\rm u}$ of $X$ (Corollary~\ref{uo-complete3}).

We refer to ~\cite{AB:06,MN:91} for all unexplained terminology and standard
facts on vector and Banach lattices. All vector lattices in this paper are
assumed to be Archimedean. 
Recall that a net $(x_\alpha)_{\alpha\in\Gamma}$ in a
vector lattice $X$ is said to \term{converge in order} to $x\in X$, written as
$x_\alpha\xrightarrow{\rm o}x$, if there exists another net
$(a_\gamma)_{\gamma\in \Lambda}$ in $X$ satisfying $a_\gamma\downarrow 0$ and
for any $\gamma\in\Lambda$ there exists $\alpha_0\in \Gamma$ such that
$\abs{x_\alpha-x}\leq a_\gamma$ for all $\alpha\geq \alpha_0$. A net
$(x_\alpha)$ is said to \term{converge in unbounded order} (uo-converge for
short) to $x\in X$, written as $x_\alpha\xrightarrow{\rm uo}x$, if
$\abs{x_\alpha-x}\wedge y\xrightarrow{\rm o}0$ for any $y\in X_+$. A linear
functional $f$ on $X$ is \term{strictly positive} if $f(x)>0$ whenever $x>0$.
A linear functional $f$ on $X$ is
\term{order continuous} if $f(x_\alpha)\rightarrow 0$ whenever
$x_\alpha\xrightarrow{\rm o}0$. Denote by $X_n^\sim$ the collection of all order
continuous functionals on $X$. Recall that a vector lattice $X$ is said to be
\term{order complete} (or Dedekind complete) if every order bounded subset of
$X$ has a supremum. It is well-known that in an order complete vector lattice,
an order bounded net $(x_\alpha)$ converges in order to $0$ if and only if
$\inf_\alpha \sup_{\beta\ge \alpha}\abs{x_\beta}=0$.  
Recall from \cite[Theorem~2.24]{AB:06} that every Archimedean vector lattice $X$ has
a (unique up to lattice isomorphisms) \term{order completion}  $X^\delta$, which is order complete and contains $X$ as a majorizing and
order dense sublattice. Finally, recall that a vector lattice is said to have the
\term{countable sup property} if any subset having a supremum admits an at most
countable subset having the same supremum.

\section{Results}\label{uo}

Recall that a net $(x_\alpha)$ in a vector lattice $X$ is said to be \term{order
Cauchy} if the doubly indexed net
$(x_\alpha-x_{\alpha'})_{(\alpha,{\alpha'})}$ order converges to $0$, and is
said to be \term{uo-Cauchy} if $(x_\alpha-x_{\alpha'})_{(\alpha,{\alpha'})}$
uo-converges to $0$. The proof of the following lemma is essentially contained
in that of \cite[Proposition~5.7]{GTX:16}.

\begin{lemma}\label{order-cauchy}
Every order bounded positive increasing net in a vector lattice $X$ is order
Cauchy.
\end{lemma}

\begin{proof}
Let $(x_{\alpha})$ be a net in $X$ such that $0 \leq x_{\alpha}\uparrow \leq x$
for some $x$. Then $(x_\alpha)$ is order convergent in the order completion
$X^\delta$ of $X$,  and is therefore, order Cauchy in $X^\delta$. It follows
from \cite[Proposition~1.5]{AS:05} (see also \cite[Corollary~2.9]{GTX:16}) that
$(x_\alpha)$ is order Cauchy in $X$.
\end{proof}

This fact is useful at times. For example, it implies that if every order
Cauchy net is order convergent then the vector lattice is order complete. It is
well-known that the
converse is also true (cf.~\cite[Proposition~2.3]{GTX:16}). Thus, the name of
order completeness is
justified: a vector lattice is  \emph{order complete} if and only if every order
Cauchy net is order
convergent. It also gives a short proof of the fact that every order continuous
Banach lattice is order complete. Indeed, if $0 \leq x_{\alpha} \uparrow \leq
x$, then $(x_\alpha)$ being order Cauchy, it is also norm Cauchy, and is
therefore,
norm convergent. The norm limit is also the supremum of $(x_\alpha)$
(\cite[Theorem~3.46]{AB:06}).

These remarks justify the importance of the following norm bounded version of
Lemma~\ref{order-cauchy}. We first establish the following technical lemma,
which generalizes \cite[Corollary~3.5]{GTX:16} and will be used in the proof
of Theorem~\ref{uo-cauchy}.

\begin{lemma}\label{uo-simple}
Let $X$ be a vector lattice and $D$ be a
set in $X_+$. The following are equivalent.
\begin{enumerate}
\item\label{uo-1} The band generated by $D$ is $X$. 
\item\label{uo-2} For any $x\in X_+$, $x \wedge d=0$ for all $d\in D$ implies
$x=0$. 
\item\label{uo-3} For any net $(x_\alpha)$ in $X_+$, 
$x_\alpha\wedge d \xrightarrow{\mathrm{o}} 0$ for any $d\in D$ implies $x_\alpha
\xrightarrow{\rm uo} 0$.
\end{enumerate}
\end{lemma}

\begin{proof}
Notice that the band $B$ generated by $D$ is $X$ iff $B^{\rm d}=\{0\}$. Since
$B^{\rm d}=D^{\rm d}$,  the equivalence of \eqref{uo-1} and \eqref{uo-2} follows
immediately.
Suppose now \eqref{uo-2} holds. Since $X$ is order dense in $X^\delta$, it is
easily seen that if $x \wedge d=0$
for some $0\leq x\in X^\delta$ and all $d\in D$ then $x=0$. Moreover, by
\cite[Corollary~2.9 and Theorem~3.2]{GTX:16}, order and uo-convergence of a net
in $X$ can pass
freely between $X$ and $X^\delta$ if the limit is in $X$. Thus for the
implication \eqref{uo-2}$\Rightarrow$\eqref{uo-3}, we may
assume, by passing to $X^\delta$, that $X$ is order complete.
Pick any $x\in X_+$. Since $x_\alpha\wedge d\xrightarrow{\rm o}0$, we have
$$[\inf_{\alpha} \sup_{\beta \ge \alpha} (x_\beta\wedge x)]\wedge
d=\inf_{\alpha} \sup_{\beta \ge \alpha} (x_\beta\wedge x\wedge
d)=[\inf_{\alpha}\sup\limits_{\beta \ge \alpha}(x_\beta \wedge d)] \wedge
x=0.$$ 
Therefore, $\inf_\alpha\sup_{\beta\geq \alpha}x_\beta \wedge x=0$,
i.e.~$x_\alpha\wedge x\xrightarrow{\rm o}0$. It follows that
$x_\alpha\xrightarrow{\rm uo}0$. This proves
\eqref{uo-2}$\Rightarrow$\eqref{uo-3}.
Suppose now \eqref{uo-3} holds but \eqref{uo-2} fails. Then we can take $x>0$
such that $x\perp d$ for all $d\in D$. Put $x_n=x$ for all $n\geq 1$. Then
$x_n\wedge d\xrightarrow{\rm o}0$ for all $d\in D$ but $x_n\xrightarrow{\rm
uo}x\neq 0$, contradicting \eqref{uo-3}. This proves 
\eqref{uo-3}$\Rightarrow$\eqref{uo-2}.
\end{proof}

\begin{theorem}\label{uo-cauchy}
Let $(x_\alpha) $ be a norm bounded positive increasing net in a Banach lattice
$X$. Suppose that
$X^{\thicksim}_n$ separates points of $X$ (this is satisfied if $X$ is order
continuous). Then $(x_\alpha)$ is uo-Cauchy in
$X$.
\end{theorem}

\begin{proof}
Since $X$ is order dense and majorizing in $X^\delta$, it follows from
\cite[Theorem~1.65]{AB:06} that each order continuous functional on 
$X$ extends uniquely to an order continuous functional on $X^\delta$. By order
density of $X$ in $X^\delta$, it is easily seen that each order continuous
functional on $X^\delta$ comes from such an extension, and moreover,
$(X^\delta)_n^\sim$ also
separates points of $X^\delta$. It is also known that the norm of $X$ extends to
a complete norm on $X^\delta$ (\cite[pp.~27]{AA:02}). Thus, by passing to
$X^\delta$, we may assume that $X$ is order complete.
For each $0\leq f\in X_n^\sim$, put $N_f:=\{x\in X:{f}(\abs{x})=0\}$, the null
ideal of $f$, and put $C_f:=N_f^{\mathrm{d}}$, the carrier of $f$. Put
$D=\cup_{0\leq f\in X^{\thicksim}_n} ({C}_f)_+$. If $x\wedge d=0$ for all $d\in
D$, then $x\geq 0$ and $x\in (C_f)^{\mathrm{d}}= N_f$ for all $0\leq f\in
X_n^\sim$, so that
$f(x)=0$ for all $0\leq f\in X_n^\sim$. It follows that $x=0$. Thus
$D$ satisfies the  statement of Lemma~\ref{uo-simple}\eqref{uo-2}.

Now pick $d\in D$, say, $d\in C_f$ for some $0\leq f\in X_n^\sim$. For each
$x\in
C_f$, put $\norm{x}_L=f(\abs{x})$. Since $f$ is strictly positive on $C_f$,
$\norm{\cdot}_L$ is a norm. It is well-known that the norm completion,
$\widetilde{C_f}$, of $(C_f,\norm{\cdot}_L)$ is some $L_1(\mu)$-space, and that
$C_f$ sits in $\widetilde{C_f}$ as
an ideal (cf.~\cite[Section~4]{GTX:16}). Let $P_f$ be the band projection for
$C_f$.
Then $\sup_\alpha \norm{P_fx_\alpha}_L=\sup_\alpha f(P_fx_\alpha)=\sup_\alpha
f(x_\alpha)<\infty$. Being increasing in $\widetilde{C_f}$ which is a KB space,
$(P_fx_\alpha)$ is convergent in order
(and also in norm) in $\widetilde{C_f}$. It follows that $(P_fx_\alpha)$ is
order Cauchy, and therefore, uo-Cauchy, in $\widetilde{C_f}$. It is uo-Cauchy in
$C_f$
and in $X$, by \cite[Theorem~3.2]{GTX:16}. It follows that
$$\abs{x_\alpha-x_{\alpha'}}\wedge d=P_f(\abs{x_\alpha-x_{\alpha'}}\wedge
d)=\abs{P_f x_\alpha-P_fx_{\alpha'}}\wedge d\xrightarrow{\rm o}0$$ in $X$. By
Lemma ~\ref{uo-simple}, $(x_\alpha)$ is uo-Cauchy in $X$.
\end{proof}

\begin{problem}
Can one remove the assumption that $X_n^\sim$ separates points of $X$ in
Theorem~\ref{uo-cauchy}?
\end{problem}

The following corollary may have appeared in the literature. Recall first that a
Banach function space over a  measure space $(\Omega,\Sigma,\mu)$
is an ideal in $L_0(\Omega,\Sigma,\mu)$ endowed with a complete lattice norm.

\begin{corollary}\label{ael}
Let $X$ be a Banach function space over a $\sigma$-finite measure space. Then
any norm bounded positive increasing sequence in $X$ converges
a.e.~to a real-valued measurable function.
\end{corollary}

\begin{proof}
By Lozanovsky's Theorem (\cite[Theorem~5.25]{AA:02}), $X_n^\sim$ separates
points of $X$. By \cite[Proposition~3.1]{GTX:16}, being uo-Cauchy for a sequence
in a function space is the same as being almost everywhere convergent. Thus the
corollary follows from Theorem~\ref{uo-cauchy}.
\end{proof}

Recall that a vector lattice $X$ is said to be \term{laterally complete} if
every collection of mutually
disjoint positive vectors admits a supremum. A vector lattice is said to be
\term{universally complete} if it
is order complete and laterally complete.  For a vector lattice $X$, denote by
$X^{\mathrm{u}}$ its universal completion (cf.~\cite[Definition~23.19 and
Theorem~23.20]{AB:78}). It is proved in \cite[Corollary~3.12]{GTX:16} that a
uo-null (resp.~uo-Cauchy) sequence in $X$ is o-null (resp.~o-convergent) in
$X^{\rm {u}}$. The following example shows that this result in general fails
for nets.

\begin{example}
Put $X=\ell_\infty$. Then $(\ell_\infty)^{\rm u}=\mathbb{R}^{\mathbb{N}}$, the
vector
lattice of all real sequences. Direct the index set $\Gamma:=\mathbb{N}\times
\mathbb{N}$ by $(m',n')\geq (m,n)$ iff $m'\geq m$ and $n'\geq n$. 
We define a net $(x_\alpha)_{\alpha\in\Gamma}$ in $\ell_\infty$ as follows: For
any
$(m,n)\in \mathbb{N} \times \mathbb{N}$ and any $k\ge 1$, put
\begin{equation*}
x_{(m,n)}(k)=\left\{
\begin{array}{r@{\; ,\;}l}
m+n  &m \text{ or } n\leq k\\
\frac{1}{m+n} &  m \text{ and } n> k
\end{array}
\right.
\end{equation*}
This net is as desired:
\begin{enumerate}
\item For any $(m,n)\in\Gamma $, we have $x_{(m,n)}(k)=m+n$ for $k$ large
enough; therefore, $x_{(m,n)}\in \ell_\infty$.
\item For $k\geq 1$, denote by $e_k$ the sequence which is equal $1$ in the
$k$-th spot and $0$ elsewhere. We have $x_{(m,n)}(k)=\frac{1}{m+n}$ for $(m,n)$
large enough. Therefore, $x_{(m,n)}\wedge e_k=\frac{1}{m+n}e_k \xrightarrow{\rm
o} 0$ in
$\mathbb{R}^\mathbb{N}$ as $(m,n)\rightarrow\infty$. Since $D:=\{e_k\}_1^\infty$
satisfies the statement of Lemma~\ref{uo-simple}\eqref{uo-2}, it follows that
$x_{(m,n)}\xrightarrow{\rm uo}0$ in $\mathbb{R}^\mathbb{N}$ as
$(m,n)\rightarrow\infty$.
\item Finally, we claim that $(x_\alpha)$ has no tail which is order bounded in
$\mathbb{R}^\mathbb{N}$, and therefore, is not order convergent in
$\mathbb{R}^\mathbb{N}$. Suppose, otherwise, that some tail
$(x_{(m,n)})_{(m,n)\geq (m_0,n_0)}$ is order bounded by some $x\in
\mathbb{R}^\mathbb{N}$. Then for $k_0=m_0+n_0+1$, we have that
$(x_{(m,n)}(k_0))_{(m,n)\geq (m_0,n_0)}$ is bounded by $x(k_0)$, which is
absurd, since $x_{(m_0,n)}(k_0)=m_0+n\rightarrow\infty$.
\end{enumerate}
\end{example}

We show, however, that an uo-Cauchy net in $X$ is \emph{uo-convergent} in
$X^{\mathrm{u}}$.
A vector lattice is said to be \term{uo-complete} if any uo-Cauchy net is
uo-convergent. A closely related notion of bounded uo-completeness has been
studied in \cite{GX:14,GTX:16}.
The following simple lemma is direct verification and we omit the proof.

\begin{lemma}\label{uo-cauchy-equal}
 A net $(x_\alpha)$ in an order complete vector lattice $X$ with a weak unit
$u>0$ is uo-Cauchy iff $\inf_{\alpha}\sup_{\beta, \beta' \ge \alpha}
\abs{x_\beta-x_{\beta'}} \wedge u =0$.
\end{lemma}

\begin{proposition}\label{uo-complete}
Let $X$ be an order complete vector lattice. If $X$ is uo-complete, then it is
universally complete. Conversely, if $X$ is universally complete, and, in
addition, has the countable sup property, then it is uo-complete.
\end{proposition}

\begin{proof}
Suppose that $X$ is uo-complete and $A$ is a collection of mutually disjoint
positive vectors in $X$. Let
$\Gamma$ be the collection of all nonempty finite subsets of $A$, directed
upward by inclusion. For any $\alpha=\{a_1,..., a_n\} \in \Gamma$, put
$x_\alpha=a_1 \vee ... \vee a_n=\sum_1^na_k$. Then $(x_\alpha)$ is an increasing
net. 
We claim that $(x_\alpha)$ is uo-Cauchy in $X$. Indeed, let $D=A\cup
(B_A^{\mathrm d})_+$ where $B_A$ is the band generated by $A$. It is clear that
$D$ satisfies the  statement of Lemma~\ref{uo-simple}\eqref{uo-2}. Moreover, for
$0\leq d\in
B_A^{\mathrm d}$, we have $\abs{x_\alpha-x_{\alpha'}}\wedge d=0$, so that
$\abs{x_\alpha-x_{\alpha'}}\wedge d \xrightarrow{\rm o} 0$ in $X$; and for $a\in
A$, we have $|x_\alpha-x_{\alpha'}| \wedge a=0$ for all $(\alpha,\alpha') \ge
(\{a\}, \{a\})$, from which it also follows that
$\abs{x_\alpha-x_{\alpha'}}\wedge a\xrightarrow{\rm o}0$. This proves the claim
by Lemma~\ref{uo-simple}. Since $X$ is uo-complete,  there exists a vector $x\in
X$ such that $x_\alpha\xrightarrow{\rm uo} x$, from which it follows easily that
$x=\sup x_\alpha=\sup A$. 

Conversely, suppose that $X$ is universally complete and has the countable sup
property. Note that $X$ has a weak unit, say, $u>0$
(\cite[Theorem~23.2]{AB:78}).
Let $(x_\alpha)$ be a uo-Cauchy net in $X$. 
In view of Lemma~\ref{uo-cauchy-equal}, we can find, by the countable sup
property,
a sequence $(\alpha_n)$ such that 
$$\inf_{n\ge 1}\sup_{\beta, \beta' \ge \alpha_n}
\abs{x_\beta-x_{\beta'}} \wedge u =0.\eqno(1)$$
By reselecting, we can make $\alpha_n$'s increasing.
It is clear that
$$\inf_{n\ge 1}\sup_{k,\ell \ge n}\abs {x_{\alpha_k}-x_{\alpha_\ell}} \wedge u
=0,$$
i.e., $(x_{\alpha_n})$ is uo-Cauchy in $X$ by \cite[Corollary~3.5]{GTX:16}.
By~\cite[Theorem~3.10]{GTX:16},  there exists $x\in X$ such that
$x_{\alpha_n}\xrightarrow{\rm uo}x$, i.e., 
$$\inf_{n\ge 1}\sup_{m\ge n} \abs{x_{\alpha_m}-x} \wedge u =0.\eqno(2)$$
By (1) and (2), it follows from
$\abs{x_\beta-x}\wedge u\leq \abs{x_\beta-x_{\alpha_n}}\wedge
u+\abs{x_{\alpha_n}-x}\wedge u$ for any $\beta\geq \alpha_n$ that
$$\inf_{n\ge 1}\sup_{\beta \ge \alpha_n}|x_\beta-x|\wedge u=0.$$
Therefore, $\inf_{\alpha}\sup_{\beta \ge \alpha}\abs{x_\beta -x }\wedge u=0$,
and $x_\alpha \xrightarrow{\rm uo} x$ in $X$.
\end{proof}

\begin{lemma}\label{cou-sup}
Let $X$be a vector lattice with a weak unit $u>0$. If $X$ has the
countable sup property, then $X^{\mathrm{u}}$ also has the countable sup
property.
\end{lemma}

\begin{proof}
It follows from order denseness of $X$ in $X^\delta$ that $u$ is also a weak
unit in $X^\delta$.
Since $X$ is order dense and majorizing in $X^\delta$, it is easily seen that
$X^\delta$ also has the countable sup property(\cite[Theorem~32.9]{LZ:71}).
Thus, since $X$ and $X^\delta$ have the same universal completion
(\cite[Theorem~23.18]{AB:78}), we may assume, by passing to $X^\delta$, that $X$
is order complete. Then $X$ is
an ideal of $X^\mathrm{u}$ (cf.~\cite[Theorem~2.31]{AB:06}).
Observe again that since $X$ is order dense in $X^{\rm u}$, $u$ is also a weak
unit in $X^\mathrm{u}$.

Now let $ A$ be a subset of $ (X^{\mathrm{u}})_+$ which has a supremum $a \in
X^\mathrm{u}$. For each
$n\geq 1$, we have $A\wedge nu:=\{a'\wedge nu:a'\in A\}\subset X$ and $a\wedge
nu\in X$. Moreover, since $a\wedge nu$ is the supremum of $A\wedge nu$ in
$X^\mathrm{u}$, it is also the supremum of $A\wedge nu$ in $X$. By the countable
sup
property of $X$, there exists a countable subset $A_n$ of $A$ such that $a\wedge
nu$ is the supremum of $A_n\wedge nu$ in $X$.

 We claim that the countable subset
$A_0:=\cup_nA_n$  satisfies $\sup A_0=a$. Indeed, the supremum of $A_0\wedge nu$
in
$X^\mathrm{u}$ is at least the supremum of $A_n\wedge nu$ in $X^\mathrm{u}$,
which is equal to $a\wedge nu$, the supremum of $A_n\wedge nu$ in $X$, since $X$
is an ideal in $X^\mathrm{u}$. Therefore, the supremum of $A_0$ in
$X^\mathrm{u}$ is at least $a\wedge nu$ for each $n\geq 1$, and is therefore at
least $a$, as $u$ is a weak unit in $X^\mathrm{u}$. The reverse inequality $\sup
A_0\leq a$ is obvious.
\end{proof}

\begin{theorem}\label{uo-complete2}
Let $X$ be a vector lattice, and $D$ be a maximal collection of disjoint
positive nonzero vectors in $X$. Suppose that the band $B_d$ generated by $d$
has the countable sup property for each $d\in D$. Then $X^{\mathrm{u}}$ is
uo-complete, and every
uo-Cauchy net in $X$ is uo-convergent in $X^{\mathrm{u}}$.
\end{theorem}

\begin{proof}
By \cite[Theorem~3.2]{GTX:16}, every uo-Cauchy net in $X$ is uo-Cauchy in
$X^\mathrm{u}$. Thus it suffices to show that $X^\mathrm{u}$ is uo-complete.
Let $(x_\alpha)$ be a uo-Cauchy net in
$X^\mathrm{u}$. Since $\abs{x^\pm-y^\pm}\leq \abs{x-y}\leq
\abs{x^+-y^+}+\abs{x^--y^-}$, we may assume without loss of generality that 
$x_\alpha$'s and $x$ are positive.

For each $d\in D$, let $B_d'$ be the band generated by $d$ in $X^\mathrm{u}$. 
It is easily seen that $B'_d$ is universally complete. 
We claim that $B_d$ is order dense in $B_d'$. Indeed, for any $0\leq w\in B_d'$,
since $X$ is order dense in $X^\mathrm{u}$, there exists a net $(x_\alpha)$ in
$X_+$ such that $0\leq x_\alpha\uparrow w$ in $X^\mathrm{u}$. For any $y\in X_+$
with $y\wedge d=0$ in $X$, we have $y\perp B_d'$ in $X^\mathrm{u}$, and thus
$y\wedge w=0$. 
It follows that $ x_\alpha\perp y$, and thus $x_\alpha\in \{d\}^{\rm dd}=B_d$.
This proves the claim.
By \cite[Theorem~23.18]{AB:78}, $B'_d$ is the universal completion of $B_d$, and
therefore, by Lemma~\ref{cou-sup}, has the countable sup property.

Let $P_d$ be
the band projection from $X^\mathrm{u}$ onto $B_d'$. Then $(P_dx_\alpha)$ is
uo-Cauchy in $B_d'$ by \cite[Lemma~3.3]{GX:14}. By
Proposition~\ref{uo-complete}, one can find $0\leq x_d\in B_d'$ such that
$\abs{P_dx_\alpha-x_d}\xrightarrow{\rm uo}0$ in $B_d'$, and thus, in
$X^\mathrm{u}$ by \cite[Theorem~3.2]{GTX:16}. Since the $d$'s are disjoint, we
know that $B_d'$'s, and therefore, $x_d$'s, are disjoint. Let $x\geq 0$ be the
supremum of $x_d$'s in $X^\mathrm{u}$. It is clear that $P_dx=x_d$. Now we have
\begin{align*}\abs{x_\alpha-x}\wedge d=&P_d(\abs{x_\alpha-x}\wedge
d)=\abs{P_dx_\alpha-P_dx}\wedge d\\
=&\abs{P_dx_\alpha-x_d}\wedge d\xrightarrow{\rm
o}0\end{align*} in $X^\mathrm{u}$. Since $X$ is order dense in $X^\mathrm{u}$,
we have
that
$D$ is also a maximal collection of  disjoint positive nonzero vectors in
$X^\mathrm{u}$, and thus by Lemma~\ref{uo-simple}, $x_\alpha\xrightarrow{\rm
uo}x$ in $X^u$. 
\end{proof}

\begin{problem}
Can one remove the countable sup property assumption in
Proposition~\ref{uo-complete} and Theorem~\ref{uo-complete2}?
\end{problem}

\begin{corollary}\label{uo-complete3}
Let $X$ be a vector lattice such that $X^\sim_n$ separates points of $X$ (this
is satisfied if $X$ is an order continuous Banach lattice).
Then $X^{\mathrm{u}}$ is uo-complete, and every
uo-Cauchy net in $X$ is uo-convergent in $X^{\mathrm{u}}$.
\end{corollary}

\begin{proof}
Let $A=\{f_\alpha:\alpha\in\Lambda\}$ be a maximal collection of  pairwise
disjoint nonzero positive functionals in $X^\sim_n$. For each $\alpha\in
\Lambda$, let $D_\alpha$ be a maximal collection of pairwise disjoint nonzero
positive vectors in the carrier $C_\alpha$ of $f_\alpha$. Put $D=\cup_{\alpha
\in \Lambda} D_\alpha$.  
Since $f_\alpha$ acts as a strictly positive functional on $C_\alpha$, each
$C_\alpha$ has the countable sup property by
~\cite[Theorem~2.6]{AB:78}. It follows that the principal band $B_d$ has the countable sup
property for each $d\in D$.
In view of Theorem~\ref{uo-complete2}, it  suffices to show that $D$ is a maximal
collection of pairwise disjoint nonzero positive vectors in $X$. Let $x>0$
such that $x\perp D$. Then for each $\alpha\in\Lambda$, $x\perp D_\alpha$, and
thus $x\perp C_\alpha$; for otherwise there would exist $0<y\in C_\alpha$ such that
$x\wedge y>0$, but $x\wedge y\in C_\alpha$ and $x\wedge y\perp D_\alpha$, contradicting the maximality of $D_\alpha$.
Therefore, $$B_x\perp C_\alpha$$ for each $\alpha$.
Since $X^\sim_n$
separates points of $X$, there exists $0\leq f\in X^\sim_n$ such that $f(x)>0$.
By ~\cite[Theorem~1.22]{AB:06}, there exists a positive linear functional $g$ on
$X$ such that $0\leq g\leq f$, $f(x)=g(x)$ and $g(y)=0$ for all $y\perp x$. Clearly, $0<g\in X_n^\sim$, and $$ C_g\subset B_x.$$
Therefore, $C_g\perp C_\alpha$ for each $\alpha$.
By \cite[Theorem~1.67]{AB:06}, $g\perp f_\alpha$ for each $\alpha\in \Lambda$, contradicting the maximality of $A$.
\end{proof}

\noindent \textbf{Acknowledgement.} The authors thank Dr.~Niushan Gao for many
valuable discussions and thank the reviewers for carefully reading the paper and
providing many suggestions. The first author is grateful to Dr.~Niushan Gao for
his guidance.

\end{document}